\documentclass[11pt]{article}
\usepackage{amsmath,amsfonts,amssymb,amsthm,graphics,amscd}

\newtheorem{theorem}{Theorem}[section]
\newtheorem{lemma}{Lemma}[section]
\newtheorem{corollary}{Corollary}[section]

\theoremstyle{definition}
\newtheorem{definition}{Definition}[section]

\def\B{{\cal B}}

\def\sbmatrix{\left[\begin{array}}
\def\endsbmatrix{\end{array}\right]}

\begin{document}

\title{Weighted weak group inverse for Hilbert space operators}

\author{Dijana Mosi\' c\footnote{This author is supported by the Ministry of Education, Science and Technological Development,
Republic of Serbia, grant no. 174007.}, Daochang
Zhang
\footnote{This author is supported by the National Natural
Science Foundation of China (NSFC) (No. 61672149; No. 51507029;
No. 61503072), and the Scientific and Technological Research
Program Foundation of Jilin Province, China (No.
JJKH20190690KJ).}}

\date{}

\maketitle

\begin{abstract} We present the weighted weak group inverse,
which is a new generalized inverse of operators between two
Hilbert spaces, introduced to extend weak group inverse for square
matrices. Some characterizations and representations of the
weighted weak group inverse are investigated. We also apply these
results to define and study the weak group inverse for a Hilbert
space operator. Using the weak group inverse, we define and
characterize various binary relations.

\medskip


\end{abstract}


\section{Introduction}

Throughout this paper, let $\B(X,Y)$ be the set of all bounded
linear operators from $X$ to $Y$, where $X$ and $Y$ are
infinite-dimensional complex Hilbert space. In the case that
$X=Y$, we set $\B(X)=\B(X,X)$. For $A\in\B(X,Y)$, $A^*$, $N(A)$,
$R(A)$ and $\sigma(A)$ represent the adjoint of $A$, the null
space, the range and the spectrum of $A$, respectively. We call
$P\in\B(X)$ an idempotent if $P^2=P$, and an orthogonal projector
if $P^2=P=P^*$. If $L$ and $M$ are closed subspaces, we denote by
$P_{L,M}$ an idempotent on $L$ along $M$, and by $P_L$ an
orthogonal projector onto $L$.

Let $A\in\B(X,Y)\backslash\{0\}$. There always exists
$B\in\B(Y,X)\backslash\{0\}$ such that $BAB=B$, which is not
unique in general and it is called an outer inverse of $A$. The
outer inverse is uniquely determined if we fix its range and
kernel. For a subspace $T$ of $X$ and a subspace $S$ of $Y$, the
unique outer inverse $B$ of $A$ with the prescribed range and the
null space $S$ will be denoted by $A^{(2)}_{T,S}$. We now present
some special classes of outer inverses.

For a fixed operator $W\in\B(Y,X)\backslash\{0\}$, an operator
$A\in\B(X,Y)$ is called Wg--Drazin invertible \cite{AK} if there
exists a unique operator $B\in\B(X,Y)$ (denoted by $A^{d,w}$) such
that
$$AWB=BWA,\quad BWAWB=B\quad{\rm and}\quad A-AWBWA{\rm\ is\ quasinilpotent}.$$
In the case that $X=Y$ and $W=I$, then $A^d=A^{d,W}$ is the
generalized Drazin inverse (or the Koliha-Drazin inverse) of $A$
\cite{K3}. We use $\B(X,Y)^{d,W}$ and $\B(X)^d$, respectively, to
denote the sets of all Wg--Drazin invertible operators in
$\B(X,Y)$ and generalized Drazin invertible operators in $\B(X)$.

The W-weighted Drazin inverse is a particular case of the
Wg--Drazin inverse for which $A-AWBWA$ is nilpotent. The Drazin
inverse is a special case of the generalized Drazin inverse for
which $A-A^2B$ is nilpotent (or equivalently $A^{k+1}B=A^k$, for
some non-negative integer $k$). The smallest such $k$ is called
the index of $A$ and it is denoted by $ind(A)$. In the case that
$ind(A)\leq 1$, $A$ is group invertible and the group inverse
$A^\#$ of $A$ is a special case of a Drazin inverse. The Drazin
inverse is very useful, and its applications in automatics,
probability, statistics, mathematical programming, numerical
analysis, game theory, econometrics, control theory and so on, can
be found in \cite{Ca1,SM-gi}. For more recent results related to
generalized Drazin inverse, W-weighted Drazin inverse and Drazin
inverse see \cite{RMSD,WMS,WYLD,ZMG}.

Prasad and Mohana \cite{PM} introduced the core--EP inverse for a
square matrix of arbitrary index, as a generalization of the core
inverse restricted to a square matrix of index one \cite{BT-core}.
The core--EP inverse was presented for generalized Drazin
invertible operators on Hilbert spaces in \cite{MDjgDMP}.

As a generalization of the core--EP inverse of a square matrix to
a rectangular matrix, the weighted core--EP inverse was given in
\cite{FLT}. In \cite{DM-weighted-core-EP}, the weighted core--EP
inverse was defined for a $W${\it g}-Drazin invertible bounded
linear operator between two Hilbert spaces, extending the concepts
of the weighted core--EP inverse for a rectangular matrix
\cite{FLT}.

Let $W\in\B(Y,X)\backslash\{0\}$ and let $A\in\B(X,Y)$ be $W${\it
g}-Drazin invertible. Then there exists the unique operator $B$
which satisfies conditions
$$WAWB=P_{R((WA)^d)}\quad{\rm and}\quad R(B)\subseteq R((AW)^d)$$ and it is
called the weighted core--EP inverse of $A$, denoted by
$A^{\tiny\textcircled{\tiny d},W}$. If $X=Y$ and $W=I$, then
$A^{\tiny\textcircled{\tiny d}}=A^{\tiny\textcircled{\tiny d},W}$
is the core--EP inverse  of $A$ \cite{MDjgDMP}. In the case that
$A\in\B(X)$ and $ind(A)\leq 1$, the core-EP inverse of $A$ is the
core inverse of $A$, denoted by $A^{\tiny\textcircled{\tiny \#}}$.
Recently, many results concerning the weighted core--EP and
core--EP inverse appeared in papers
\cite{FLT-max,GaoChen,GaoChenPatricio1,DM-core-EP-pre-order,Wang-core-EP,ZCLW}.

In \cite{WC-weak-group}, the weak group inverse was recently
defined for square matrices of an arbitrary index and presented as
a generalization of the group inverse.

We extend the definition of the weak group inverse of a square
matrix to a Wg--Drazin invertible bounded linear operator between
two Hilbert spaces and present a new generalized inverse, named
the weighted weak group inverse. We obtain some properties of the
weighted weak group inverse, in particular, an operator matrix
representation, characterizations and representations of the
weighted weak group inverse. As an application of these results,
we present and characterize the weak group inverse of a
generalized Drazin invertible bounded linear operator on a Hilbert
space. Using the weak group inverse, we define and study several
binary relations.

\section{Weighted weak group inverse}

In order to define the weighted weak group inverse of a Wg--Drazin
invertible bounded linear operator between two Hilbert spaces as
an extension of the weak group inverse of a square matrix, we need
following auxiliary result.

\begin{lemma}{\rm \cite{DM-weighted-core-EP}}\label{te0-repres-A-W} Let $W\in\B(Y,X)\backslash\{0\}$ and $A\in\B(X,Y)^{d,W}$.
Then \begin{equation}A=\sbmatrix{cc}
A_1&A_2\\0&A_3\endsbmatrix:\sbmatrix{c}
R((WA)^d)\\N[((WA)^d)^*]\endsbmatrix\rightarrow\sbmatrix{c}
R((AW)^d)\\N[((AW)^d)^*]\endsbmatrix\label{jed-oper-A}\end{equation}
and
\begin{equation}W=\sbmatrix{cc} W_1&W_2\\0&W_3\endsbmatrix:\sbmatrix{c}
R((AW)^d)\\N[((AW)^d)^*]\endsbmatrix\rightarrow\sbmatrix{c}
R((WA)^d)\\N[((WA)^d)^*]\endsbmatrix,\label{jed-oper-W}\end{equation}
where $A_1\in\B(R((WA)^d),R((AW)^d))^{-1}$,
$W_1\in\B(R((AW)^d),R((WA)^d))^{-1}$,\linebreak
$A_3W_3\in\B(N[((AW)^d)^*])^{qnil}$ and
$W_3A_3\in\B(N[((WA)^d)^*])^{qnil}$.
In addition,
\begin{equation}A^{d,W}=\sbmatrix{cc}
(W_1A_1W_1)^{-1}&W_1^{-1}U\\0&0\endsbmatrix:\sbmatrix{c}
R((WA)^d)\\N[((WA)^d)^*]\endsbmatrix\rightarrow\sbmatrix{c}
R((AW)^d)\\N[((AW)^d)^*]\endsbmatrix,\label{oper-WgD-A}\end{equation}
\begin{equation}A^{\tiny\textcircled{\tiny d},W}=\sbmatrix{cc} (W_1A_1W_1)^{-1}&0\\0&0\endsbmatrix:\sbmatrix{c}
R((WA)^d)\\N[((WA)^d)^*]\endsbmatrix\rightarrow\sbmatrix{c}
R((AW)^d)\\N[((AW)^d)^*]\endsbmatrix\label{jed-oper-Bwcore},\end{equation}
$$(AW)^d=\sbmatrix{cc}
(A_1W_1)^{-1}&T\\0&0\endsbmatrix:\sbmatrix{c}
R((AW)^d)\\N[((AW)^d)^*]\endsbmatrix\rightarrow\sbmatrix{c}
R((AW)^d)\\N[((AW)^d)^*]\endsbmatrix,$$
$$(WA)^d=\sbmatrix{cc}
(W_1A_1)^{-1}&U\\0&0\endsbmatrix:\sbmatrix{c}
R((WA)^d)\\N[((WA)^d)^*]\endsbmatrix\rightarrow\sbmatrix{c}
R((WA)^d)\\N[((WA)^d)^*]\endsbmatrix,$$ where
$$T=\sum\limits_{n=0}^\infty
(A_1W_1)^{-(n+2)}(A_1W_2+A_2W_3)(A_3W_3)^n$$ and
$$U=\sum\limits_{n=0}^\infty
(W_1A_1)^{-(n+2)}(W_1A_2+W_2A_3)(W_3A_3)^n.$$
\end{lemma}


We first give algebraic definition of a new generalized inverse.

\begin{theorem}\label{te1-w-weak-group} Let $W\in\B(Y,X)\backslash\{0\}$ and $A\in\B(X,Y)^{d,W}$.
Then the system of equations
\begin{equation}AWBWB=B\quad{ and}\quad AWB=A^{\tiny\textcircled{\tiny d},W}WA \label{def1-w-weak-group}\end{equation} is consistent and it
has the unique solution given by
\begin{equation} B=\sbmatrix{cc}
(W_1A_1W_1)^{-1}&(A_1W_1)^{-2}(A_2+W_1^{-1}W_2A_3)\\0&0\endsbmatrix:\sbmatrix{c}
R((WA)^d)\\N[((WA)^d)^*]\endsbmatrix\rightarrow\sbmatrix{c}
R((AW)^d)\\N[((AW)^d)^*]\endsbmatrix\label{jed-oper-B-w-weak-group},\end{equation}
where $A$ and $W$ are represented as in {\rm (\ref{jed-oper-A})}
and {\rm (\ref{jed-oper-W})}, respectively.
\end{theorem}

\begin{proof} Using
(\ref{jed-oper-A}), (\ref{jed-oper-W}), (\ref{jed-oper-Bwcore})
and (\ref{jed-oper-B-w-weak-group}), we get $$AWB=\sbmatrix{cc}
W_1^{-1}&(A_1W_1)^{-1}A_2+(W_1A_1W_1)^{-1}W_2A_3\\0&0\endsbmatrix=A^{\tiny\textcircled{\tiny
d},W}WA$$ and $AWBWB=B$, that is, $B$ is a solution of the system
(\ref{def1-w-weak-group}).

If an operator $B$ satisfies (\ref{def1-w-weak-group}), then
$$B=(AWB)WB=A^{\tiny\textcircled{\tiny d}}W(AWB)
=A^{\tiny\textcircled{\tiny d},W}WA^{\tiny\textcircled{\tiny
d},W}WA=(A^{\tiny\textcircled{\tiny d},W}W)^2A$$ and so $B$ is the
unique solution of the system (\ref{def1-w-weak-group}).
\end{proof}

For $X=Y$ and $W=I$ in Theorem \ref{te1-w-weak-group}, we get the
next consequence.

\begin{corollary}\label{cor1-weak-group} Let $A\in\B(X)^d$.
Then the system of equations
\begin{equation}AB^2=B\quad{ and}\quad AB=A^{\tiny\textcircled{\tiny d}}A \label{def1-weak-group}\end{equation} is consistent and it
has the unique solution given by
\begin{equation} B=\sbmatrix{cc}
A_1^{-1}&A_1^{-2}A_2\\0&0\endsbmatrix:\sbmatrix{c}
R(A^d)\\N[(A^d)^*]\endsbmatrix\rightarrow\sbmatrix{c}
R(A^d)\\N[(A^d)^*]\endsbmatrix\label{jed-oper-B-weak-group},\end{equation}
where \begin{equation} A=\sbmatrix{cc}
A_1&A_2\\0&A_3\endsbmatrix:\sbmatrix{c} R(A^d)\\
N((A^d)^*)\endsbmatrix\rightarrow\sbmatrix{c}
R(A^d)\\
N((A^d)^*)\endsbmatrix,\label{jed-oper-A123}\end{equation}
$A_1\in\B(R(A^d))$ is invertible and $A_3\in\B[N((A^d)^*)]$ is
quasinilpotent.
\end{corollary}

%

\begin{definition} Let $W\in\B(Y,X)\backslash\{0\}$ and $A\in\B(X,Y)^{d,W}$. The weighted weak group inverse of $A$ is defined as
$$A^{\otimes,W}=(A^{\tiny\textcircled{\tiny d},W}W)^2A.$$
\end{definition}

\begin{definition} Let $A\in\B(X)^d$. The weak group inverse of $A$ is defined as
$$A^{\otimes}=(A^{\tiny\textcircled{\tiny d}})^2A.$$
\end{definition}

Remark that, by Lemma \ref{te0-repres-A-W} and Theorem
\ref{te1-w-weak-group}, the weighted weak group inverse is
different from the $W${\it g}-Drazin inverse and weighted core--EP
inverse. Hence, the weighted weak group and weak group inverses
provide new classes of generalized inverses for operators. If $X$
and $Y$ are finite dimensional, then, for
$W\in\B(Y,X)\backslash\{0\}$, every operator $A\in\B(X,Y)$ has the
weighted weak group inverse. For $A\in\B(X)$ and $ind(A)\leq 1$,
the weak group inverse of $A$ is the group inverse of $A$.

We also have a definition of the weighted weak group inverse from
a geometrical point of view.

\begin{theorem} Let $W\in\B(Y,X)\backslash\{0\}$ and $A\in\B(X,Y)^{d,W}$. The system of conditions
\begin{equation}WAWB=P_{R(WA^{d,W}), N(A^{\tiny\textcircled{\tiny d},W}WA)}\qquad{and}\qquad
R(B)\subseteq R(A^{d,W}) \label{def2-w-weak-group}\end{equation}
is consistent and it has the unique solution
$B=(A^{\tiny\textcircled{\tiny d},W}W)^2A$.
\end{theorem}

\begin{proof} For $B=(A^{\tiny\textcircled{\tiny d},W}W)^2A$, we have
that $WAWB=WA^{\tiny\textcircled{\tiny d},W}WA$ is a projector
onto $R(WA^{\tiny\textcircled{\tiny d},W})=R(WA^{d,W})$ along
$N(WA^{\tiny\textcircled{\tiny
d},W}WA)=N(A^{\tiny\textcircled{\tiny d},W}WA)$ and $R(B)\subseteq
R(A^{\tiny\textcircled{\tiny d},W})=R(A^{d,W})$, i.e. $B$
satisfies conditions (\ref{def2-w-weak-group}).

Assume that two operators $B_1$ and $B_2$ satisfy conditions
(\ref{def2-w-weak-group}). Firstly, $WAW(B_1-B_2)=P_{R(WA^{d,W}),
N(A^{\tiny\textcircled{\tiny d},W}WA)}-P_{R(WA^{d,W}),
N(A^{\tiny\textcircled{\tiny d},W}WA)}=0$ implies
$R(B_1-B_2)\subseteq N(WAW)\subseteq N(A^{d,W}WAW)$. Further,
$R(B_1)\subseteq R(A^{d,W})=R(A^{d,W}WAW)$ and $R(B_2)\subseteq
R(A^{d,W}WAW)$ give $R(B_1-B_2)\subseteq R(A^{d,W}WAW)\cap
N(A^{d,W}WAW)=\{0\}$. Therefore, $B_1=B_2$ and only $B$ satisfies
(\ref{def2-w-weak-group}).
\end{proof}

Consequently, the geometrical approach is given now for the weak
group inverse.

\begin{corollary} Let $A\in\B(X)^d$. The system of conditions
\begin{equation}AB=P_{R(A^d),N(A^{\tiny\textcircled{\tiny d}}A)}\qquad{and}\qquad
R(B)\subseteq R(A^d), \label{def2-weak-group}\end{equation} is
consistent and it has the unique solution
$B=(A^{\tiny\textcircled{\tiny d}})^2A$.
\end{corollary}

%

We consider some idempotents determined by the weighted weak group
inverse and observe that the weighted weak group inverse is an
outer inverse.

\begin{lemma}\label{le1-w-weak-group} Let $W\in\B(Y,X)\backslash\{0\}$ and $A\in\B(X,Y)^{d,W}$. Then:
\begin{itemize}
\item [\rm (i)] $AWA^{\otimes,W}W$ is a projector onto
$R(A^{d,W})$ along $N(A^{\tiny\textcircled{\tiny d},W}WAW)$;

\item [\rm (ii)] $WAWA^{\otimes,W}$ is a projector onto
$R(WA^{d,W})$ along $N(A^{\tiny\textcircled{\tiny d},W}WA)$;

\item [\rm (iii)] $A^{\otimes,W}WAW$ is a projector onto
$R(A^{d,W})$ along $N(A^{\tiny\textcircled{\tiny d},W}W(AW)^2)$;

\item [\rm (iv)] $WA^{\otimes,W}WA$ is a projector onto
$R(WA^{d,W})$ along $N(A^{\tiny\textcircled{\tiny d},W}(WA)^2)$;

\item [\rm (v)]
$A^{\otimes,W}=(WAW)^{(2)}_{R(A^{d,W}),N(A^{\tiny\textcircled{\tiny
d},W}WA)}$.
\end{itemize}
\end{lemma}

\begin{proof} (i) Notice that $AWA^{\otimes,W}W=A^{\tiny\textcircled{\tiny d},W}WAW$ is a projector onto
$R(A^{\tiny\textcircled{\tiny d},W})=R(A^{d,W})$ along
$N(A^{\tiny\textcircled{\tiny d},W}WAW)$.

(v) Since $A^{\otimes,W}=(A^{\tiny\textcircled{\tiny d},W}W)^2A$
and $AWA^{\otimes,W}=A^{\tiny\textcircled{\tiny d},W}WA$, we
obtain
$$A^{\otimes,W}WAWA^{\otimes,W}=(A^{\tiny\textcircled{\tiny
d},W}W)^2AWA^{\tiny\textcircled{\tiny
d},W}WA=(A^{\tiny\textcircled{\tiny d},W}W)^2A=A^{\otimes,W}.$$
Hence, $A^{\otimes,W}$ is an outer inverse of $WAW$ with
$R(A^{\otimes,W})=R((A^{\tiny\textcircled{\tiny d},W}W)^2A)$ and
$N(A^{\otimes,W})=N((A^{\tiny\textcircled{\tiny d},W}W)^2A)$. From
$$A^{d,W}=A^{\tiny\textcircled{\tiny d},W}WA^{d,W}WA
=A^{\tiny\textcircled{\tiny d},W}WA^{\tiny\textcircled{\tiny
d},W}WA^{d,W}WAWA=(A^{\tiny\textcircled{\tiny
d},W}W)^2(AW)^2A^{d,W},$$ we get $R(A^{d,W})\subseteq
R((A^{\tiny\textcircled{\tiny d},W}W)^2A)\subseteq R(A^{d,W})$
which yields $R(A^{\otimes,W})=R(A^{d,W})$. By (\ref{jed-oper-A}),
(\ref{jed-oper-W}) and (\ref{jed-oper-Bwcore}), we show that
$A(WA^{\tiny\textcircled{\tiny d},W})^2=A^{\tiny\textcircled{\tiny
d},W}$. Now, we have that
$$A^{\tiny\textcircled{\tiny
d},W}WA=AW(A^{\tiny\textcircled{\tiny
d},W}WA^{\tiny\textcircled{\tiny d},W}WA)=AWA^{\otimes,W}$$ which
gives $N(A^{\otimes,W})\subseteq N(A^{\tiny\textcircled{\tiny
d},W}WA)\subseteq N(A^{\otimes,W})$, i.e.
$N(A^{\otimes,W})=N(A^{\tiny\textcircled{\tiny d},W}WA)$. Thus,
$A^{\otimes,W}=(WAW)^{(2)}_{R(A^{d,W}),N(A^{\tiny\textcircled{\tiny
d},W}WA)}$.

Similarly, we verify parts (ii)--(iv).
\end{proof}

By Lemma \ref{le1-w-weak-group}, notice that the weak group
inverse of $A$ is an outer inverse of $A$.

\begin{corollary} Let $A\in\B(X)^d$. Then:
\begin{itemize}
\item [\rm (i)] $AA^{\otimes}$ is a projector onto $R(A^{d})$
along $N(A^{\tiny\textcircled{\tiny d}}A)$;

\item [\rm (ii)] $A^{\otimes}A$ is a projector onto $R(A^{d})$
along $N(A^{\tiny\textcircled{\tiny d}}A^2)$;

\item [\rm (iii)]
$A^{\otimes}=A^{(2)}_{R(A^d),N(A^{\tiny\textcircled{\tiny d}}A)}$.
\end{itemize}
\end{corollary}

Several characterizations of the weighted weak group inverse are
presented now.

\begin{theorem}\label{te3-w-weak-group} Let $W\in\B(Y,X)\backslash\{0\}$ and $A\in\B(X,Y)^{d,W}$.
Then, for $B\in\B(X,Y)$, the following statements are equivalent:
\begin{itemize}
\item[\rm(i)] $B$ is the weighted weak group inverse of $A$;

\item[\rm(ii)] $B$ satisfies
$$A^{\tiny\textcircled{\tiny d},W}WAWB=B\quad and \quad AWB=A^{\tiny\textcircled{\tiny d},W}WA;$$

\item[\rm(iii)] $B$ satisfies
$$BWAWB=B,\quad AWB=A^{\tiny\textcircled{\tiny d},W}WA\quad and\quad BWA^{\tiny\textcircled{\tiny d},W}=
A^{\tiny\textcircled{\tiny d},W}WA^{\tiny\textcircled{\tiny
d},W};$$

\item[\rm(iv)] $B$ satisfies
$$BWAWB=B,\quad AWB=A^{\tiny\textcircled{\tiny d},W}WA\quad and\quad BWA^{d,W}=
A^{d,W}WA^{d,W}.$$

%
\end{itemize}
\end{theorem}

\begin{proof} (i) $\Rightarrow$ (ii) $\wedge$ (iii): The equality $B=(A^{\tiny\textcircled{\tiny
d},W}W)^2A$ gives $B=A^{\tiny\textcircled{\tiny
d},W}WAW(A^{\tiny\textcircled{\tiny
d},W}W)^2A=A^{\tiny\textcircled{\tiny d},W}WAWB$ and
$BWA^{\tiny\textcircled{\tiny d},W}=(A^{\tiny\textcircled{\tiny
d},W}W)^2AWA^{\tiny\textcircled{\tiny
d},W}=A^{\tiny\textcircled{\tiny d},W}WA^{\tiny\textcircled{\tiny
d},W}$. The rest is clear.

(ii) $\Rightarrow$ (i): It follows by
$B=A^{\tiny\textcircled{\tiny
d},W}W(AWB)=A^{\tiny\textcircled{\tiny
d},W}WA^{\tiny\textcircled{\tiny d},W}WA=A^{\otimes,W}$.

(iii) $\Rightarrow$ (i): We have that
$B=BW(AWB)=(BWA^{\tiny\textcircled{\tiny
d},W})WA=A^{\tiny\textcircled{\tiny
d},W}WA^{\tiny\textcircled{\tiny d},W}WA=A^{\otimes,W}$.

(iii) $\Leftrightarrow$ (iv): Using $A^{\tiny\textcircled{\tiny
d},W}=A^{d,W}WAWA^{\tiny\textcircled{\tiny d},W}$ and
$A^{d,W}=A^{\tiny\textcircled{\tiny d},W}WAWA^{d,W}$, we obtain
these equivalences.
%
\end{proof}

Applying Theorem \ref{te3-w-weak-group}, we can characterize the
weak group inverse in the following way.

\begin{corollary}\label{cor4-weak-group} Let $A\in\B(X)^d$.
Then, for $B\in\B(X)$, the following statements are equivalent:
\begin{itemize}
\item[\rm(i)] $B$ is the weak group inverse of $A$;

\item[\rm(ii)] $B$ satisfies
$$A^{\tiny\textcircled{\tiny d}}AB=B\quad and \quad AB=A^{\tiny\textcircled{\tiny d}}A;$$

\item[\rm(iii)] $B$ satisfies
$$BAB=B,\quad AB=A^{\tiny\textcircled{\tiny d}}A\quad and\quad BA^{\tiny\textcircled{\tiny d}}=
A^{\tiny\textcircled{\tiny d}}A^{\tiny\textcircled{\tiny d}};$$

\item[\rm(iv)] $B$ satisfies
$$BAB=B,\quad AB=A^{\tiny\textcircled{\tiny d}}A\quad and\quad BA^{d}=
(A^{d})^2.$$
\end{itemize}
\end{corollary}

In the case that $X$ and $Y$ are finite dimensional, the condition
$BWA^{d,W}= A^{d,W}WA^{d,W}$ ($BA^{d}= (A^{d})^2$) of Theorem
\ref{te3-w-weak-group}(iv) (Corollary \ref{cor4-weak-group}(iv))
can be replaced with the equivalent condition $B(WA)^{k+1}=(WA)^k$
for $k=ind(WA)$ ($BA^{k+1}=A^k$ for $k=ind(A)$).

Using idempotents and orthogonal projectors, we present some
representations of the weighted weak group inverse of $A$ in the
next theorem.

\begin{theorem}\label{te4-w-weak-group} Let $W\in\B(Y,X)\backslash\{0\}$ and $A\in\B(X,Y)^{d,W}$.
Then the following statements holds:
\begin{itemize}
\item[\rm(i)] $A^{\otimes,W}=A^{\tiny\textcircled{\tiny
d},W}P_{R(WA^{d,W}), N(A^{\tiny\textcircled{\tiny d},W}WA)}$;

\item[\rm(ii)] $A^{\otimes,W}=A^{d,W}P_{R(WA^{d,W}),
N(A^{\tiny\textcircled{\tiny d},W}WA)}$;


\item[\rm(iii)] $WP_{R((AW)^d)}$ is Moore--Penrose invertible,
$WA(WA)^{\tiny\textcircled{\tiny d}}WA$ is group invertible and
$$A^{\otimes,W}=(WP_{R((AW)^d)})^\dag(WA(WA)^{\tiny\textcircled{\tiny
d}}WA)^\#=(WP_{R((AW)^d)})^\dag(P_{R((WA)^d)}WA)^\#;$$

\item[\rm(iv)] $A^{\otimes,W}=[(AW)^2]^{\tiny\textcircled{\tiny
d}}AWA^{\tiny\textcircled{\tiny
d},W}WA=[(AW)^2]^{\tiny\textcircled{\tiny d}}AP_{R(WA^{d,W}),
N(A^{\tiny\textcircled{\tiny d},W}WA)}$;

\item[\rm(v)] $WP_{R((AW)^d)}$ is Moore--Penrose invertible and
$A^{\otimes,W}=[(AW)^d]^2(WP_{R((AW)^d)})^\dag WA$;

\item[\rm(vi)] $WP_{R((AW)^d)}$ is Moore--Penrose invertible and
$A^{\otimes,W}=(WP_{R((AW)^d)})^\dag[(WA)^2]^{\tiny\textcircled{\tiny
d}}WA$;

\item[\rm(vii)] $(AW)^3(AW)^d$ is group invertible and
$A^{\otimes,W}=[(AW)^3(AW)^d]^{\tiny\textcircled{\tiny
\#}}AWA^{\tiny\textcircled{\tiny d},W}WA$.
\end{itemize}
\end{theorem}

\begin{proof} (i) It follows by $A^{\otimes,W}=(A^{\tiny\textcircled{\tiny
d},W}W)^2A=A^{\tiny\textcircled{\tiny d},W}P_{R(WA^{d,W}),
N(A^{\tiny\textcircled{\tiny d},W}WA)}$.

(ii) We get
\begin{eqnarray*}A^{\otimes,W}&=&(A^{\tiny\textcircled{\tiny
d},W}W)^2A=A^{d,W}WAW(A^{\tiny\textcircled{\tiny
d},W}W)^2A=A^{d,W}WA^{\tiny\textcircled{\tiny
d},W}WA\\&=&A^{d,W}P_{R(WA^{d,W}), N(A^{\tiny\textcircled{\tiny
d},W}WA)}.\end{eqnarray*}

(iii) Let $A$ and $W$ be represented as in (\ref{jed-oper-A}) and
(\ref{jed-oper-W}), respectively. Notice that the orthogonal
projector $P_{R((AW)^d)}$ has the following representation:
$$P_{R((AW)^d)}=\sbmatrix{cc} I&0\\0&0\endsbmatrix:\sbmatrix{c}
R((AW)^d)\\N[((AW)^d)^*]\endsbmatrix\rightarrow\sbmatrix{c}
R((AW)^d)\\N[((AW)^d)^*]\endsbmatrix.$$ We observe that
$WP_{R((AW)^d)}=\sbmatrix{cc} W_1&0\\0&0\endsbmatrix$ is
Moore--Penrose invertible and
$$(WP_{R((AW)^d)})^\dag=\sbmatrix{cc}
W_1^{-1}&0\\0&0\endsbmatrix.$$ Because $A$ is $W${\it g}-Drazin
invertible, then $WA$ is generalized Drazin invertible,
$$WA=\sbmatrix{cc}
W_1A_1&W_1A_2+W_2A_3\\0&W_3A_3\endsbmatrix\quad{\rm
and}\quad(WA)^{\tiny\textcircled{\tiny d}}=\sbmatrix{cc}
(W_1A_1)^{-1}&0\\0&0\endsbmatrix.$$ We now have that
$WA(WA)^{\tiny\textcircled{\tiny d}}WA=\sbmatrix{cc}
W_1A_1&W_1A_2+W_2A_3\\0&0\endsbmatrix$ is group invertible and
$$(WA(WA)^{\tiny\textcircled{\tiny d}}WA)^\#=\sbmatrix{cc}
(W_1A_1)^{-1}&(W_1A_1)^{-2}(W_1A_2+W_2A_3)\\0&0\endsbmatrix.$$
Therefore,
\begin{eqnarray*}(WP_{R((AW)^d)})^\dag(WA(WA)^{\tiny\textcircled{\tiny
d}}WA)^\#&=&\sbmatrix{cc}
(W_1A_1W_1)^{-1}&(A_1W_1)^{-2}(A_2+W_1^{-1}W_2A_3)\\0&0\endsbmatrix\\&=&A^{\otimes,W}.\end{eqnarray*}

(iv) Since $A$ is $W${\it g}-Drazin invertible, then $AW$ is
generalized Drazin invertible and so $(AW)^2$ generalized Drazin
invertible too. Hence, $[(AW)^2]^{\tiny\textcircled{\tiny d}}$
exists and, using (\ref{jed-oper-A}) and (\ref{jed-oper-W}),
$$[(AW)^2]^{\tiny\textcircled{\tiny d}}=\sbmatrix{cc}
(A_1W_1)^{-2}&0\\0&0\endsbmatrix.$$ By (\ref{jed-oper-Bwcore}), we
obtain
\begin{eqnarray*}[(AW)^2]^{\tiny\textcircled{\tiny
d}}AWA^{\tiny\textcircled{\tiny d},W}WA&=&\sbmatrix{cc}
(A_1W_1)^{-2}&0\\0&0\endsbmatrix\sbmatrix{cc}
A_1&A_2+W_1^{-1}W_2A_3\\0&0\endsbmatrix\\&=&\sbmatrix{cc}
(W_1A_1W_1)^{-1}&(A_1W_1)^{-2}(A_2+W_1^{-1}W_2A_3)\\0&0\endsbmatrix\\&=&A^{\otimes,W}.
\end{eqnarray*}

Similarly, we verify parts (v)--(vi).
\end{proof}

We have new representations for the weak group inverse by Theorem
\ref{te4-w-weak-group}.

\begin{corollary}
Let $A\in\B(X)^d$. Then the following statements holds:
\begin{itemize}
\item[\rm(i)] $A^{\otimes}=A^{\tiny\textcircled{\tiny
d}}P_{R(A^{d}), N(A^{\tiny\textcircled{\tiny d}}A)}$;

\item[\rm(ii)] $A^{\otimes}=A^{d}P_{R(A^{d}),
N(A^{\tiny\textcircled{\tiny d}}A)}$;

\item[\rm(iii)] $AA^{\tiny\textcircled{\tiny d}}A$ is group
invertible and $A^{\otimes}=(AA^{\tiny\textcircled{\tiny
d}}A)^\#=(P_{R(A^d))}A)^\#;$

\item[\rm(iv)] $A^{\otimes}=(A^2)^{\tiny\textcircled{\tiny
d}}AA^{\tiny\textcircled{\tiny d}}A=(A^2)^{\tiny\textcircled{\tiny
d}}AP_{R(A^{d}), N(A^{\tiny\textcircled{\tiny d}}A)}$;

\item[\rm(v)] $A^{\otimes}=(A^d)^2P_{R(A^d)}A$;

\item[\rm(vi)] $A^{\otimes}=(A^2)^{\tiny\textcircled{\tiny d}}A$;

\item[\rm(vii)] $A^3A^d$ is group invertible and
$A^{\otimes}=(A^3A^d)^{\tiny\textcircled{\tiny
\#}}AA^{\tiny\textcircled{\tiny
d}}A=(A^3A^d)^{\tiny\textcircled{\tiny \#}}A$.
\end{itemize}
\end{corollary}

\begin{proof} Because $P_{R(A^d)}$ is an orthogonal projector,
notice that $P_{R(A^d))}$ is Moore--Penrose invertible and
$(P_{R(A^d)})^\dag=P_{R(A^d)}$. Then we can easily verify this
result.
\end{proof}

Recall that $A\in\B(X,Y)^{d,W}$ with $A^{d,W}=B\in\B(X,Y)$ if and
only if $AW\in\B(Y)^{d}$ with $(AW)^d=BW$ if and only if
$WA\in\B(X)^{d}$ with $(WA)^d=WB$ \cite{AK}. For the weighted
core--EP inverse, the situation is similar in the case of the
operator $WA$, but it is a little different for $AW$
\cite{DM-weighted-core-EP}. We see now that the weighted weak
group inverse acts as the weighted core--EP inverse.

\begin{theorem}\label{te5-w-weak-group} Let $W\in\B(Y,X)\backslash\{0\}$ and let $A\in\B(X,Y)$.
\begin{itemize}
\item[{\rm (a)}] Then the following statements are equivalent:
\begin{itemize}
\item[\rm(i)] $A$ is weighted weak group invertible with
$A^{\otimes,W}=B$;

\item[\rm(ii)] $WA$ is weak group invertible with
$(WA)^{\otimes}=WB.$
\end{itemize}
In addition,
$$A^{\otimes,W}=A[(WA)^{\otimes}]^2.$$

\item[{\rm (b)}] If $A$ is $W${\it g}-Drazin invertible, $A$ and
$W$ are represented by {\rm (\ref{jed-oper-A})} and {\rm
(\ref{jed-oper-W})}, respectively, then
\begin{itemize}
\item[\rm(i)] $(AW)^{\otimes}=A^{\otimes,W}W$ if and only if
$W_2A_3W_3=0$;

\item[\rm(ii)] $A^{\otimes,W}=[(AW)^{\otimes}]^2A$ if and only if
$A_2W_3A_3=0$.
\end{itemize}
\end{itemize}
\end{theorem}

\begin{proof} (a) (i) $\Rightarrow$ (ii): By \cite[Theorem 2.4]{DM-weighted-core-EP}, $WA^{\tiny\textcircled{\tiny
d},W}=(WA)^{\tiny\textcircled{\tiny d}}$, which implies
$WA^{\otimes,W}=W(A^{\tiny\textcircled{\tiny
d},W}W)^2A=[(WA)^{\tiny\textcircled{\tiny d}}]^2WA=(WA)^\otimes$.

(ii) $\Rightarrow$ (i): We observe that
\begin{eqnarray*}
A[(WA)^{\otimes}]^2&=&A[(WA)^{\tiny\textcircled{\tiny
d}}]^2WA[(WA)^{\tiny\textcircled{\tiny
d}}]^2WA=A[(WA)^{\tiny\textcircled{\tiny d}}]^3WA\\
&=&A(WA^{\tiny\textcircled{\tiny
d},W})^3WA=(A^{\tiny\textcircled{\tiny d},W}W)^2A=A^{\otimes,W}.
\end{eqnarray*}

(b) (i) It follows from the equalities
$$(AW)^{\otimes}=\sbmatrix{cc}
(A_1W_1)^{-1}&(A_1W_1)^{-2}(A_1W_2+A_2W_3)\\0&0\endsbmatrix$$ and
$$A^{\otimes,W}W=\sbmatrix{cc}
(A_1W_1)^{-1}&(W_1A_1W_1)^{-1}W_2+(A_1W_1)^{-2}(A_2+W_1^{-1}W_2A_3)W_3\\0&0\endsbmatrix.$$

(ii) Using (\ref{jed-oper-B-w-weak-group}) and
$$[(AW)^{\otimes}]^2A=\sbmatrix{cc}
(W_1A_1W_1)^{-1}&(A_1W_1)^{-2}(A_2+W_1^{-1}W_2A_3+(A_1W_1)^{-1}A_2W_3A_3)\\0&0\endsbmatrix,$$
we get that (ii) holds.
\end{proof}

Using the weak group inverses of $AW$ and $WA$, we obtain the next
formula for the weighted weak group inverses of $A$.

\begin{theorem} Let $W\in\B(Y,X)\backslash\{0\}$ and $A\in\B(X,Y)^{d,W}$. Then
$$A^{\otimes,W}=(AW)^{\otimes}A(WA)^{\otimes}.$$
\end{theorem}

\begin{proof} Let $A$ and $W$ be given by (\ref{jed-oper-A}) and
(\ref{jed-oper-W}), respectively. Then
\begin{eqnarray*}(AW)^{\otimes}A(WA)^{\otimes}&=&\sbmatrix{cc}
W_1^{-1}&(A_1W_1)^{-1}A_2+(A_1W_1)^{-2}(A_1W_2+A_2W_3)A_3\\0&0\endsbmatrix\\&\times&\sbmatrix{cc}
(W_1A_1)^{-1}&(W_1A_1)^{-2}(W_1A_2+W_2A_3)\\0&0\endsbmatrix\\&=&A^{\otimes,W}.\end{eqnarray*}
\end{proof}

We also investigate necessary and sufficient conditions for
$AWA^{\otimes,W}=A^{\otimes,W}WA$ to hold.

\begin{theorem}\label{te8-w-weak-group} Let $W\in\B(Y,X)\backslash\{0\}$ and $A\in\B(X,Y)^{d,W}$. If $A$ and $W$ are represented by {\rm
(\ref{jed-oper-A})} and {\rm (\ref{jed-oper-W})}, respectively,
then the following statements are equivalent:
\begin{itemize}
\item[\rm(i)] $AWA^{\otimes,W}=A^{\otimes,W}WA$;

\item[\rm(ii)] $(W_1A_2+W_2A_3)W_3A_3=0$;

\item[\rm(iii)] $[(WA)^\otimes]^2=[(WA)^2]^\otimes$.
\end{itemize}
In this case, $A^{\otimes,W}=A^{d,W}$.
\end{theorem}

\begin{proof} (i) $\Leftrightarrow$ (ii): The equalities $$AWA^{\otimes,W}=\sbmatrix{cc}
W_1^{-1}&(A_1W_1)^{-1}A_2+(W_1A_1W_1)^{-1}W_2A_3\\0&0\endsbmatrix$$
and $$A^{\otimes,W}WA=\sbmatrix{cc}
W_1^{-1}&(A_1W_1)^{-1}A_2+(W_1A_1W_1)^{-1}W_2A_3+(A_1W_1)^{-2}(A_2+W_1^{-1}W_2A_3)W_3A_3\\0&0\endsbmatrix$$
give $AWA^{\otimes,W}=A^{\otimes,W}WA$ if and only if
$(A_2+W_1^{-1}W_2A_3)W_3A_3=0$ which is equivalent to
$(W_1A_2+W_2A_3)W_3A_3=0$.

(ii) $\Leftrightarrow$ (iii): Firstly, we observe that
$$[(WA)^{\otimes}]^2=\sbmatrix{cc}
(W_1A_1)^{-2}&(W_1A_1)^{-3}(W_1A_2+W_2A_3)\\0&0\endsbmatrix.$$
Furthermore, $$(WA)^2=\sbmatrix{cc}
(W_1A_1)^{2}&W_1A_1(W_1A_2+W_2A_3)+(W_1A_2+W_2A_3)W_3A_3\\0&(W_3A_3)^2\endsbmatrix$$
is generalized Drazin invertible and, by Corollary
\ref{cor1-weak-group},
$$[(WA)^2]^\otimes=\sbmatrix{cc}
(W_1A_1)^{-2}&(W_1A_1)^{-3}(W_1A_2+W_2A_3)+(W_1A_1)^{-4}(W_1A_2+W_2A_3)W_3A_3\\0&0\endsbmatrix.$$
Thus, $[(WA)^\otimes]^2=[(WA)^2]^\otimes$ is equivalent to
$(W_1A_2+W_2A_3)W_3A_3=0$.

Notice that, by $(W_1A_2+W_2A_3)W_3A_3=0$ and (\ref{oper-WgD-A}),
we obtain $A^{d,W}=A^{\otimes,W}$.
\end{proof}

Applying Theorem \ref{te8-w-weak-group} and Theorem
\ref{te5-w-weak-group}(b)(ii), we get equivalent conditions for
$AA^{\otimes}=A^{\otimes}A$ to be satisfied. We observe that
conditions (i)--(iii) appeared for the weak group inverse of a
square matrix, but the condition (iv) is new even for matrix case.

\begin{corollary} Let $A\in\B(X)^d$. If $A$ is represented by {\rm
(\ref{jed-oper-A123})}, then the following statements are
equivalent:
\begin{itemize}
\item[\rm(i)] $AA^{\otimes}=A^{\otimes}A$;

\item[\rm(ii)]  $A_2A_3=0$;

\item[\rm(iii)] $(A^\otimes)^2=(A^2)^\otimes$;

\item[\rm(iv)] $A^\otimes=(A^\otimes)^2A$.
\end{itemize}
In this case, $A^{\otimes}=A^{d}$.
\end{corollary}

Similarly as Theorem \ref{te8-w-weak-group}, we verify the next
result.

\begin{theorem} Let $W\in\B(Y,X)\backslash\{0\}$ and $A\in\B(X,Y)^{d,W}$. If $A$ and $W$ are represented by {\rm
(\ref{jed-oper-A})} and {\rm (\ref{jed-oper-W})}, respectively,
then the following statements are equivalent:
\begin{itemize}
\item[\rm(i)] $[(AW)^\otimes]^2=[(AW)^2]^\otimes$;

\item[\rm(ii)] $(A_1W_2+A_2W_3)A_3W_3=0$.
\end{itemize}
\end{theorem}

\section{Weighted weak group relations}

Various types of partial orders and pre-orders were defined based
on various types of generalized inverses \cite{MBM,DM-GI-monog}.
We firstly introduce a new binary relation using the weak group
inverse.

\begin{definition} Let $B\in\B(X)$ and $A\in\B(X)^d$. Then we say that $A$ is
below $B$ under the weak group relation (denoted by
$A\leq^{\otimes}B$) if
$$AA^{\otimes}=BA^{\otimes}\quad{\rm and}\quad
A^{\otimes}A=A^{\otimes}B.$$
\end{definition}

Remark that the relation "$\leq^{\otimes}$" is not a partial
order, because it is not antisymmetric. Indeed, if $A,B\in\B(X)$
are quasinilpotent and $A\neq B$, then $A^d=0$ and $B^d=0$ imply
$A^{\otimes}=0$ and $B^{\otimes}=0$. Hence, $A\leq^{\otimes}B$ and
$B\leq^{\otimes}A$, but $A\neq B$.

By the following example, we see that the relation
"$\leq^{\otimes}$" is not transitive and so it is not a pre-order.

{\bf Example 3.1.} Consider a $3\times 3$ block matrices
$$A=\sbmatrix{ccc} 1&1&1\\0&0&1\\0&0&0
\endsbmatrix,\quad B=\sbmatrix{ccc} 1&1&0\\0&0&2\\0&0&0
\endsbmatrix\quad{\rm and}\quad C=\sbmatrix{ccc} 1&0&1\\0&1&1\\0&1&1
\endsbmatrix.$$ Then $$A^\otimes=\sbmatrix{ccc} 1&1&1\\0&0&0\\0&0&0
\endsbmatrix\quad{\rm and}\quad B^\otimes=\sbmatrix{ccc} 1&1&0\\0&0&0\\0&0&0
\endsbmatrix.$$ The equalities $AA^\otimes=BA^\otimes$, $A^\otimes A=A^\otimes B$,
$BB^\otimes=CB^\otimes$ and $B^\otimes B=B^\otimes C$ give
$A\leq^{\otimes}B$ and $B\leq^{\otimes}C$. Because $A^\otimes
A\neq A^\otimes C$, the relation $A\leq^{\otimes}C$ is not
satisfied and thus "$\leq^{\otimes}$" is not transitive.

\begin{lemma}\label{le3-w-weak-group} Let $B\in\B(X)$ and $A\in\B(X)^d$. Then:
\begin{itemize}
\item [\rm (i)] $AA^{\otimes}=BA^{\otimes}$ $\Leftrightarrow$
$A^{\tiny\textcircled{\tiny d}}A=B(A^{\tiny\textcircled{\tiny
d}})^2A$ $\Leftrightarrow$ $A^{\tiny\textcircled{\tiny
d}}=B(A^{\tiny\textcircled{\tiny d}})^2$ $\Leftrightarrow$
$A^d=BA^{\tiny\textcircled{\tiny d}}A^d$ $\Leftrightarrow$
$AA^d=BA^d$;

\item [\rm (ii)] $A^{\otimes}A=A^{\otimes}B$ $\Leftrightarrow$
$A^{\tiny\textcircled{\tiny d}}A^2=A^{\tiny\textcircled{\tiny
d}}AB$.
\end{itemize}
\end{lemma}

We can get characterizations of the weak group relation combining
conditions of Lemma \ref{le3-w-weak-group} from parts (i) and
(ii).

For operators between two Hilbert spaces, we consider a weighted
operator and define the following binary relations.

\begin{definition} Let $A,B\in\B(X,Y)$ and
$W\in\B(Y,X)\backslash\{0\}$. If $A$ is $W${\it g}-Drazin
invertible, then we say that
\begin{itemize}
\item[{\rm (i)}] $A\leq^{\otimes,W,r}B$ if $AW\leq^{\otimes}BW$,

\item[{\rm (ii)}] $A\leq^{\otimes,W,l}B$ if $WA\leq^{\otimes} WB$,

\item[{\rm (iii)}] $A\leq^{\otimes,W}B$ if $A\leq^{\otimes,W,r}B$
and $A\leq ^{\otimes,W,l}B$,
\end{itemize} where $\leq^{\otimes}$ is adequately considered on $\B(X)$ or $\B(Y)$.
\end{definition}

Several characterizations of the relation $\leq^{\otimes,W,r}$ are
presented now.

\begin{theorem}\label{te9-w-weak-group} Let $W\in\B(Y,X)\backslash\{0\}$,
$A\in\B(X,Y)^{d,W}$ and $B\in\B(X,Y)$. Then the following
statements are equivalent:
\begin{itemize}
\item[{\rm (i)}] $A\leq^{\otimes,W,r}B$;

\item[{\rm (ii)}] $AW(AW)^{\otimes}=BW(AW)^{\otimes}$ and
$(AW)^{\otimes}AW=(AW)^{\otimes}BW$;

\item[{\rm (iii)}] the following matrix representations with
respect to the orthogonal sums $X=R((WA)^d)\oplus N[((WA)^d)^*]$
and $Y=R((AW)^d)\oplus N[((AW)^d)^*]$ hold
$$A=\sbmatrix{cc} A_1&A_2\\0&A_3\endsbmatrix, \ W=\sbmatrix{cc} W_1&W_2\\0&W_3\endsbmatrix,
\ B=\sbmatrix{cc} A_1&B_2\\0&B_3\endsbmatrix,$$ where
$A_1\in\B(R((WA)^d),R((AW)^d))^{-1}$,
$W_1\in\B(R((AW)^d),R((WA)^d))^{-1}$,
$(A_2-B_2)W_3+(A_1W_1)^{-1}(A_1W_2+A_2W_3)(A_3-B_3)W_3=0$,
$A_3W_3\in\B(N[((AW)^d)^*])^{qnil}$ and
$W_3A_3\in\B(N[((WA)^d)^*])^{qnil}$.\end{itemize}
\end{theorem}

\begin{proof} (i) $\Leftrightarrow$ (ii): By the definition of the relation $\leq^{\otimes,W,r}$, this is clear.

(ii) $\Rightarrow$ (iii): Let $A$ and $W$ be given by
(\ref{jed-oper-A}) and (\ref{jed-oper-W}), respectively. Suppose
that
$$B=\sbmatrix{cc}
B_1&B_2\\B_4&B_3\endsbmatrix:\sbmatrix{c}
R((WA)^d)\\N[((WA)^d)^*]\endsbmatrix\rightarrow\sbmatrix{c}
R((AW)^d)\\N[((AW)^d)^*]\endsbmatrix.$$ Then $$BW=\sbmatrix{cc}
B_1W_1&B_1W_2+B_2W_3\\B_4W_1&B_4W_2+B_3W_3\endsbmatrix,$$
$$AW(AW)^{\otimes}=\sbmatrix{cc}I&(A_1W_1)^{-1}(A_1W_2+A_2W_3)\\0&0\endsbmatrix$$
and $$BW(AW)^{\otimes}=\sbmatrix{cc}
B_1W_1(A_1W_1)^{-1}&B_1W_1(A_1W_1)^{-2}(A_1W_2+A_2W_3)\\B_4W_1(A_1W_1)^{-1}&B_4W_1(A_1W_1)^{-2}(A_1W_2+A_2W_3)\endsbmatrix.$$
Therefore, $AW(AW)^{\otimes}=BW(AW)^{\otimes}$ is equivalent to
$B_1=A_1$ and $B_4=0$. Further, the equalities
$$(AW)^{\otimes}AW=
\sbmatrix{cc}I&(A_1W_1)^{-1}(A_1W_2+A_2W_3)+(A_1W_1)^{-2}(A_1W_2+A_2W_3)A_3W_3\\0&0\endsbmatrix,$$
$$(AW)^{\otimes}BW=\sbmatrix{cc}
I&(A_1W_1)^{-1}(A_1W_2+B_2W_3)+(A_1W_1)^{-2}(A_1W_2+A_2W_3)B_3W_3\\0&0\endsbmatrix$$
and $(AW)^{\otimes}AW=(AW)^{\otimes}BW$ give
$(A_2-B_2)W_3+(A_1W_1)^{-1}(A_1W_2+A_2W_3)(A_3-B_3)W_3=0$.

(iii) $\Rightarrow$ (ii): This part can be checked by direct
computations.
\end{proof}

In an analogy way, we characterize of the relation $\leq
^{\tiny\textcircled{\tiny d},W,l}$.

\begin{theorem}\label{te10-w-weak-group} Let $W\in\B(Y,X)\backslash\{0\}$,
$A\in\B(X,Y)^{d,W}$ and $B\in\B(X,Y)$. Then the following
statements are equivalent:
\begin{itemize}
\item[{\rm (i)}] $A\leq^{\otimes,W,l}B$;

\item[{\rm (ii)}] $WA(WA)^{\otimes}=WB(WA)^{\otimes}$ and
$(WA)^{\otimes}WA=(WA)^{\otimes}WB$;

\item[{\rm (iii)}] $WAWA^{\otimes,W}=WBWA^{\otimes,W}$ and
$WA^{\otimes,W}WA=WA^{\otimes,W}WB$;

\item[{\rm (iv)}] the following matrix representations with
respect to the orthogonal sums $X=R((WA)^d)\oplus N[((WA)^d)^*]$
and $Y=R((AW)^d)\oplus N[((AW)^d)^*]$ hold
$$A=\sbmatrix{cc} A_1&A_2\\0&A_3\endsbmatrix, \ W=\sbmatrix{cc} W_1&W_2\\0&W_3\endsbmatrix,
\ B=\sbmatrix{cc} A_1-W_1^{-1}W_2B_4&B_2\\B_4&B_3\endsbmatrix,$$
where $A_1\in\B(R((WA)^d),R((AW)^d))^{-1}$ and
$W_1\in\B(R((AW)^d),R((WA)^d))^{-1}$, $W_3B_4=0$,
$B_2=A_2+W_1^{-1}W_2(A_3-B_3)+(W_1A_1W_1)^{-1}(W_1A_2+W_2A_3)W_3(A_3-B_3)$,
$A_3W_3\in\B(N[((AW)^d)^*])^{qnil}$ and
$W_3A_3\in\B(N[((WA)^d)^*])^{qnil}$.\end{itemize}
\end{theorem}

Combining Theorem \ref{te9-w-weak-group} and Theorem
\ref{te10-w-weak-group}, we get the next results.

\begin{corollary} Let $W\in\B(Y,X)\backslash\{0\}$,
$A\in\B(X,Y)^{d,W}$ and $B\in\B(X,Y)$. Then the following
statements are equivalent:
\begin{itemize}
\item[{\rm (i)}] $A\leq^{\otimes,W}B$;

\item[{\rm (ii)}] the following matrix representations with
respect to the orthogonal sums $X=R((WA)^d)\oplus N[((WA)^d)^*]$
and $Y=R((AW)^d)\oplus N[((AW)^d)^*]$ hold
$$A=\sbmatrix{cc} A_1&A_2\\0&A_3\endsbmatrix, \ W=\sbmatrix{cc} W_1&W_2\\0&W_3\endsbmatrix,
\ B=\sbmatrix{cc} A_1&B_2\\0&B_3\endsbmatrix,$$ where
$A_1\in\B(R((WA)^d),R((AW)^d))^{-1}$,
$W_1\in\B(R((AW)^d),R((WA)^d))^{-1}$,
$B_2=A_2+W_1^{-1}W_2(A_3-B_3)+(W_1A_1W_1)^{-1}(W_1A_2+W_2A_3)W_3(A_3-B_3)$,
$W_2A_3W_3(A_3-B_3)W_3=0$, $A_3W_3\in\B(N[((AW)^d)^*])^{qnil}$ and
$W_3A_3\in\B(N[((WA)^d)^*])^{qnil}$.\end{itemize}
\end{corollary}

\begin{corollary} Let
$A\in\B(X)^d$ and $B\in\B(X)$. Then the following statements are
equivalent:
\begin{itemize}
\item[{\rm (i)}] $A\leq^{\otimes}B$;

\item[{\rm (ii)}] the following matrix representations with
respect to the orthogonal sum $X=R(A^d)\oplus N[(A^d)^*]$ hold
$$A=\sbmatrix{cc} A_1&A_2\\0&A_3\endsbmatrix,
\ B=\sbmatrix{cc}
A_1&A_2+A_1^{-1}A_2(A_3-B_3)\\0&B_3\endsbmatrix,$$ where
$A_1\in\B(R(A^d))$ is invertible, $A_3\in\B(N[(A^d)^*])$ is
quasinilpotent.
\end{itemize}
\end{corollary}

Dijana Mosi\'c

\bigskip

Faculty of Sciences and Mathematics, University of Ni\v s, P.O.
Box 224, 18000 Ni\v s, Serbia

\bigskip

{\it E-mail:} {\tt dijana@pmf.ni.ac.rs}

\bigskip

Daochang Zhang

\bigskip
College of Sciences, Northeast Electric Power University, Jilin,
P.R. China.

\bigskip
{\it E-mail:} {\tt daochangzhang@126.com}

\end{document}